\documentclass[11pt,letterpaper,reqno]{amsart}

\usepackage{amsmath,amsthm,amssymb,tikz-cd,comment,hyperref,cleveref,breqn,float,caption,subfig,ytableau,tikz,cases,enumerate}

\allowdisplaybreaks
\theoremstyle{plain}
\newtheorem{theorem}{Theorem}[section]
\newtheorem*{theorem*}{Theorem}
\newtheorem{lemma}[theorem]{Lemma}
\newtheorem{proposition}[theorem]{Proposition}

\newtheorem{thmalphabetintro}{Theorem}

\theoremstyle{definition}
\newtheorem{definition}[theorem]{Definition}
\newtheorem{example}[theorem]{Example}
\newtheorem{convention}[theorem]{Convention}

\newtheorem{remark}[theorem]{Remark}

\newtheorem*{question*}{Question}

\newtheorem*{problem*}{Problem}

\newtheorem{open problem}[theorem]{Open Problem}
\newtheorem{remark and notation}[theorem]{Remark and Notation}
\newtheorem{remark and definition}[theorem]{Remark and Definition}
\newtheorem{definition and notation}[theorem]{Definition and Notation}
\newtheorem{notation and convention}[theorem]{Notation and Convention}
\newtheorem{notation and remark}[theorem]{Notation and Remark}
\newtheorem{convention and notation}[theorem]{Convention and Notation}

\def \p {\mathbb{P}}

\def \z {\mathbb{Z}}
\def \r {\mathbb{R}}
\def \c {\mathbb{C}}

\def \i {\mathcal{I}}

\def \l {\mathcal{L}}

\def \H {\textnormal{H}}

\def \im {\operatorname{im}}

\def \codim {\operatorname{codim}}

\def \tor {\operatorname{Tor}}
\def \reg {\operatorname{reg}}

\def \py {\operatorname{py}}
\def \qp {\operatorname{qp}}

\def \re {\operatorname{re}}
\def \im {\operatorname{im}}

\def \spec {\operatorname{Spec}}

\def \det {\operatorname{det}}

\def \Sing {\operatorname{Sing}}
\def \ver {\operatorname{Vert}}

\title[{Quadratic persistence and Pythagoras numbers of varieties}]{On quadratic persistence and Pythagoras numbers of totally real projective varieties}

\author{Jong In Han}
\address{Jong In Han, School of Mathematics, Korea Institute for Advanced Study (KIAS), 85 Hoegi-ro, Dongdaemun-gu, Seoul 02455, Republic of Korea}
\email{jihan09@kias.re.kr}

\author{Jaewoo Jung}
\address{Jaewoo Jung, Global Basic Research Laboratory - Algebra and Geometry of Spaces of Tensors, and Applications (GBRL-AGSTA), Daegu Gyeongbuk Institute of Science and Technology (DGIST), 333 Techno Jungang-daero, Hyeonpung-eup, Dalseong-gun, Daegu 42988, Republic of Korea}
\email{jaewoojung@dgist.ac.kr}

\author{Euisung Park}
\address{Euisung Park, Department of Mathematics, Korea University, Seoul 136-701, Republic of Korea}
\email{euisungpark@korea.ac.kr}

\thanks{}

\begin{document}
\begin{abstract}
In this paper, we study the relationship between quadratic persistence and the Pythagoras number of totally real projective varieties. Building upon the foundational work of Blekherman et al. in \cite{MR4397034}, we extend their characterizations of arithmetically Cohen-Macaulay varieties with next-to-maximal quadratic persistence to arbitrary case. Our main result classifies totally real non-aCM varieties of codimension $c$ and degree $d$ that exhibit next-to-maximal quadratic persistence in the cases where $c=3$ and $d \geq 6$ or $c \geq 4$ and $d \geq 2c+3$. We further investigate the quadratic persistence and Pythagoras number in the context of curves of maximal regularity and linearly normal smooth curves of genus 3.
\end{abstract}

\maketitle

\section{Introduction}
Let $X$ be a nondegenerate projective variety in the projective space $\p_{\c}^r = \p^r$.
If the set $X(\r)$ of real points in $X$ is Zariski dense in $X$, i.e., $X$ is \emph{totally real}, then generators of the saturated homogeneous ideal $I(X)$ of $X$ can be chosen to be polynomials with real coefficients. For totally real varieties, we associate the variety $X$ with the ring $R_X := \r[x_0, \dots, x_r]/(I(X) \cap \r[x_0, \dots, x_r])$.
The \emph{Pythagoras number}, denoted by $\py(X)$, of $X$ is defined as the smallest positive integer $t$ such that any sum of squares of linear forms in $R_X$ can be expressed as a sum of at most $t$ squares.

In \cite{MR4397034}, the authors introduced a new algebraic invariant called \textit{quadratic persistence} of $X$ which provides a lower bound of the Pythagoras number of $X$.
For any subset $\Gamma$ of $k$ linearly independent points in $\p^r$, let $\pi_\Gamma:\p^{r}\dashrightarrow \p^{r-k}$ denote the linear projection map from $\Gamma$.
The quadratic persistence $\qp(X)$ of $X \subset \p^r$ is defined as the smallest integer $k$ such that there exists a subset $\Gamma \subset X$ of $k$ linearly independent points for which the ideal $I({\pi_\Gamma(X)})$ contains no quadratic forms.
In \cite{MR4397034}, the authors found fundamental connections between the semi-algebraic invariant $\py(X)$ and various algebraic invariants of $X$ such as quadratic persistence, Green-Lazarsfeld index, and the smallest dimension of variety of minimal degree containing $X$.
For example, if $X$ is totally real, then
\begin{equation}\label{equ:py,qp}\tag{$\star$}
    r + 1 - \qp(X) \le \py(X).
\end{equation}
Based on such results concerning $\py(X)$, they obtained complete classifications and another characterizations of varieties with minimal Pythagoras number and arithmetically Cohen-Macaulay varieties with next to minimal Pythagoras number:
\begin{thmalphabetintro}\cite[Theorem 1.4]{MR4397034}\label{thm:BSSV1}
Suppose $X$ is a nondegenerate totally real variety in $\p^r$.
    Then the following are equivalent:
    \begin{enumerate}[\rm (i)]
        \item $\qp(X) = \codim(X)$;
        \item $\py(X) = \dim(X) + 1$;
        \item $\deg(X) = \codim(X) + 1$, i.e., $X$ is a variety of minimal degree.
    \end{enumerate}
\end{thmalphabetintro}
\begin{thmalphabetintro}\cite[Theorem 1.5]{MR4397034}\label{thm:BSSV2}
If $X$ is a nondegenerate totally real arithmetically Cohen-Macaulay (aCM) variety with $\deg(X)\geq \codim(X)+2$, then the following are equivalent:
    \begin{enumerate}[\rm (i)]
        \item $\qp(X) = \codim(X) - 1$;
        \item $\py(X) = \dim(X) + 2$;
        \item $\deg(X) = \codim(X) + 2$, or $X$ is a codimension 1 subvariety of a variety of minimal degree.     
    \end{enumerate}
\end{thmalphabetintro}

Along this line, it is natural to ask whether the statement of \Cref{thm:BSSV2} still holds without the aCM condition.
Our main goal in this paper is to extend the classification and characterization in \Cref{thm:BSSV2} to general (possibly non-aCM) varieties having next to minimal Pythagoras number.

In our first main result below, we generalize \Cref{thm:BSSV2} as follows.
Note that the Pythagoras numbers and quadratic persistence of codimension $2$ varieties are already completely known (\textit{cf.} \Cref{rem:codim_two}).
\begin{theorem}\label{mainthm}
    Let $X\subset \p^{r}$ be a totally real nondegenerate projective variety of dimension $n$, codimension $c\geq 3$, and degree $d$.
    \begin{enumerate}[\rm (1)]
        \item\label{mainthm:VAMD} If $d=c+2$, then $\qp(X) = c - 1$ and $\py(X) = n +2$.
        \item\label{mainthm:general} Suppose either $c = 3$ and $d \geq 6$, or $c \geq 4$ and $d \geq 2c+3$.
        Then the following are equivalent:
        \begin{enumerate}[\rm (i)]
            \item $\qp(X) = c-1$;
            \item $\py(X) = n+2$;
            \item $X$ is a codimension $1$ subvariety of a variety of minimal degree.
        \end{enumerate}
    \end{enumerate}
\end{theorem}

Regarding the proof of \Cref{mainthm} \eqref{mainthm:VAMD}, when X is aCM, \Cref{thm:BSSV2} implies that $\qp(X) = c-1$ and $\py(X) = n+2$. 
If X is not aCM, it was shown in \cite{MR2274517} that there always exists an $(n+1)$-dimensional variety $Y$ of minimal degree containing $X$.
By combining this fact with the results of \cite{MR4397034}, we can still conclude that $\qp(X) = c-1$, and the above inequality yields $\py(X) \geq  n+2$. 
The most delicate part of the proof is to show that $\py(X) \leq n+2$. 
To achieve this, we prove ``totally real $\mathcal{K}_{p,1}$-theorem" (\Cref{lem:totally_real_Kp1}), which guarantees that $Y$ is totally real.

In \Cref{mainthm} \eqref{mainthm:general}, the implications $\rm{(iii)}\Rightarrow\rm{(ii)}\Rightarrow \rm{(i)}$ can be proven for all $d \geq c+3$ by using the totally real $\mathcal{K}_{p,1}$-theorem and the results of \cite{MR4397034}, respectively. 
The most difficult part of the proof is establishing the implication $\rm{(i)}\Rightarrow\rm{(iii)}$, which consists of two main steps. 
We first handle the case $c = 3$, and then complete the proof for the case $c \geq 4$ and $d\geq 2c+3$ by solving the following \emph{Lifting Problem} in such a case. \\

\noindent\textbf{Lifting Problem.}
    Let $X$ be a nondegenerate projective variety in $\p^{r}$ of dimension $n$, codimension $c$, and degree $d\geq c+3$.
    If $X_p = \pi_p(X)$ in $\p^{r-1}$ is contained in an $(n+1)$-dimensional variety of minimal degree for general $p\in X$, then $X$ is contained in an $(n+1)$-dimensional variety of minimal degree.\\

In fact, if Lifting Problem could be solved for all $d  \geq c+3$, then \Cref{mainthm} \eqref{mainthm:general} would hold for all such $d$ and hence the aCM assumption can be removed from \Cref{thm:BSSV2} (\textit{cf.} \Cref{transform_to_lp}).
We believe that introducing this perspective represents a significant step toward a complete solution to the problem.

For the remainder of this paper, it is worth mentioning that all known cases satisfy equality in \eqref{equ:py,qp}.
In this context, we investigate the quadratic persistence and Pythagoras number of two kinds of curves to test the sharpness of \eqref{equ:py,qp}: \\
\begin{enumerate}
    \item[(a)] curves of maximal regularity;
    \item[(b)] linearly normal smooth curves of genus at most three. \\
\end{enumerate}

First, let $\mathcal{C} \subset \p^r$ be a nondegenerate projective curve of degree $d \geq r+2$. In \cite{MR0704401}, the authors proved that the Castelnuovo-Mumford regularity of $\mathcal{C}$, denoted by $\reg(\mathcal{C})$, is always less than or equal to $d-r+2$. Due to \cite{MR1992539}, $\mathcal{C}$ is called a \textit{curve of maximal regularity} if $\reg(\mathcal{C})=d-r+2$.

The following is our second result.
\begin{theorem}\label{mainthm:cmr}
    Let $\mathcal{C} \subset \mathbb{P}^r$ $(r \geq 4)$ be a totally real curve of maximal regularity with degree $d \geq r + 2$.
    Then $\qp(\mathcal{C})\in \{r-3, r-2\}$ and 
    $$\py(\mathcal{C}) = r + 1 - \qp(\mathcal{C}).$$ 
    Moreover, the following are equivalent:
    \begin{enumerate}[\rm (i)]
        \item $\qp(\mathcal{C}) = r-2$;
        \item $\py(\mathcal{C})=3$;
        \item $\mathcal{C}$ is contained in a surface of minimal degree.
    \end{enumerate}
\end{theorem}
 
Next, we consider a linearly normal smooth curve $\mathcal{C} \subset \p^r$ of genus $g$ and degree $d$. If $\mathcal{C}$ is hyperelliptic, then it is always contained in a surface of minimal degree, and thus, by our totally real $\mathcal{K}_{p,1}$-theorem, we obtain 
$\qp(\mathcal{C}) = r - 2$ and $\py(\mathcal{C}) = 3$. Therefore, the first open case occurs when $\mathcal{C}$ is nonhyperelliptic and $g=3$. For this case, we obtain the following result.
\begin{theorem}\label{thm:nonhyperelliptic curve of genus 3}
Let $\mathcal{C}\subset \p^{r}$ $(r\geq 3)$ be a totally real linearly normal smooth nonhyperelliptic curve of genus $3$ embedded by a line bundle $\mathcal{L}$ on $\mathcal{C}$.
Then $\qp (\mathcal{C}) \in \{r-3,r-2\}$ and the following are equivalent:
\begin{enumerate}[\rm (i)]
    \item $\qp (\mathcal{C}) = r - 2$;
    \item $\py(\mathcal{C})=3$;
    \item $\mathcal{C}$ is contained in a surface of minimal degree;
    \item Either $\mathcal{L} = \omega_{C} ^2$  or else $\mathcal{L} = \omega_{C} ^{2} \otimes \mathcal{O}_{\mathcal{C}}(-p)$ for some $p \in \mathcal{C}$.
\end{enumerate}
\end{theorem}

This article is organized as follows.  
In \Cref{sec:Prelim}, we introduce essential background material related to our main objects of study.
In \Cref{sec:Codim3}, we establish explicit formulas for the quadratic persistence and the Pythagoras number of varieties of codimension three.
In \Cref{sec:mainthm}, we show the lifting problem for the case $d\geq 2c+3$. Using this together with the results in \Cref{sec:Codim3}, we prove the main theorem.
In \Cref{sec:MaxRegCurves} and \Cref{sec:nonhyper}, we investigate the quadratic persistence and the Pythagoras number of curves of maximal regularity and linearly normal smooth nonhyperelliptic curves of genus $3$, respectively.

\bigskip

\section*{Acknowledgement}
We would like to thank Grigoriy Blekherman for initiating this project and for engaging in useful discussions.
J.I. Han was partially supported by the National Research Foundation of Korea (NRF) grant funded by the Korea government (MSIT) (No. RS-2024-00352592) when he was at Korea Advanced Institute of Science and Technology (KAIST). J.I. Han is currently supported by a KIAS Individual Grant (MG101401) at Korea Institute for Advanced Study.
J. Jung was supported by the Institute for Basic Science (IBS-R032-D1-a00), and has been supported since March 2025 by the National Research Foundation of Korea (NRF) grant funded by the Korean government(MSIT) (RS-2024-00414849).
E. Park is supported by the National Research Foundation of Korea(NRF) grant funded by the Korea government(MSIT)(No. 2022R1A2C1002784). 

\bigskip

\section{Preliminaries}\label{sec:Prelim}
Throughout this article, we use the convention that varieties are integral. 
The varieties we consider are nondegenerate in most cases.
Let $X\subset\p^r_\c$ be a projective variety.
We denote by $I(X)$ its homogeneous ideal and denote by $S_X$ its homogeneous coordinate ring.

We recall the Betti numbers of a projective variety $X$ and some algebraic quantities that can be read from the Betti table of $X$.
\begin{notation and remark}\label{betti_remark}
\phantom{}
\begin{enumerate}[\rm (1)]
    \item The \textit{$(i,j)$-th graded Betti number of $X$} is $$\beta_{i,j}(X):=\dim\tor_i(S_X,\c)_{i+j}.$$
    \item The regularity $\reg(X)$ of $X$ is the smallest integer $m$ such that
    $$H^i(\p^{r},\mathcal{I}_X(m-i))=0 ~\text{for all}~ i\geq 1.$$
    By \cite{MR741934}, the regularity can be alternatively defined as
    $$\reg(X):=\min\{m\in \z\mid \beta_{i,j}(X)=0\text{ for all }i\geq 0\text{ and }j\geq m\}.$$
    \item\label{ell_def}  $\ell(X)$ is the length of the quadratic strand of the graded minimal free resolution of $X$, i.e.,
    $$\ell(X) := \min\{i\in\z_{\geq 1}\mid \beta_{i,1}(X)=0\}-1.$$
    By Green's $\mathcal{K}_{p,1}$-theorem in \cite{MR739785}, it holds that $0\leq \ell(X)\leq c$, and that $\ell(X)=c$ if and only if $d=c+1$, while $\ell(X)=c-1$ if and only if either $d=c+2$ or $d\geq c+3$ and $X$ is contained in a variety of minimal degree as a divisor (\textit{cf.} \cite[Theorem 1.1]{MR4816776}).
    \item  $a(X)$ is the Green--Lazarsfeld index of $X$, i.e.,
    $$a(X) := \max\{j\in\z_{\geq 0}\cup \{\infty\}\mid \beta_{i,2}(X)=0\text{ for all $i\leq j$} \}.$$ 
\end{enumerate}
\end{notation and remark}

In this paper, we consider two fields: $\r$ and $\c$. For the convenience, we use the following convention.

\begin{convention}
    If the field is not specified, then it is $\c$.
\end{convention}

Let $X\subset\p^r$ be a projective variety. In the following, we define totally real varieties and remark some basic properties of being totally real.

\begin{definition}
    $X$ is called \textit{totally real} if the set $X(\r)$ of real points in $X$ is Zariski dense in $X$.
\end{definition}

\begin{remark}\label{rem:tr}
    Let $\c[x_0,\cdots,x_r]$ be the polynomial ring of $\p^r_\c$.
    \begin{enumerate}[\rm (1)]
        \item\label{ideal_gen_by_real} If $X$ is totally real, then $I(X)$ is generated by real polynomials. 
        Indeed, if $X$ is totally real, for $f\in I(X)$, its complex conjugate $\overline{f}$ vanishes on $X(\r)$. 
        As $X(\r)\subset X$ is dense, it vanishes on the entire $X$.
        Hence $\overline{f}\in I(X)$ so that both $\re(f)$ and $\im(f)$ are contained in $I(X)$.
        \item When $X$ is totally real, we define $I_\r(X):=I(X)\cap \r[x_0,\cdots,x_r]$.
        Note that the extension of $I_\r(X)$ under $\r[x_0,\cdots,x_r]\to\c[x_0,\cdots,x_r]$ is $I(X)$ by (\ref{ideal_gen_by_real}). The prime ideal $I_\r(X)$ defines a geometrically integral $\r$-scheme $X_\r\subset\p^r_\r$ whose base change to $\c$ is $X=X_\r\times _{\spec\r} \spec\c$.
        \item If $X$ contains a smooth real point and $I(X)$ is generated by real polynomials, then $X$ is totally real (\textit{cf.} \cite{MR0683127}).
    \end{enumerate}
\end{remark}

On the other hand, the following lemma enables us to consider the general inner projections only when we deal with the quadratic persistence.

    \begin{lemma}\cite[Lemma 3.3]{MR4397034}\label{lem:qp}
    Let $X\subset \p_{\c}^r$ be a nondegenerate projective variety.
    Then the locus in $X^m$ on which $$(p_1,\dots,p_m)\mapsto \dim_\c I(\pi_{\{p_1,\cdots,p_m\}}(X))_2$$ achieves its minimum is Zariski open.
    \end{lemma}

    We collect the results in \cite{MR4397034} for totally real varieties that are necessary for our discussions as follows. Note that some results enumerated here are shown without assuming the irreducibility in \cite{MR4397034}.
    \begin{theorem}\cite{MR4397034}\label{thm:qpmain}
        Let $X\subset\p^r$ be a nondegenerate totally real variety.
        Then the following hold.
        \begin{enumerate}[\rm (1)]
            \item Let $\widetilde{X}$ be a totally real variety in $\mathbb{P}^r$ that contains $X$. Then 
                \begin{equation*}\py(X) \leq \py (\widetilde{X}).
                \end{equation*}
                In particular, if $\widetilde{X}$ is a variety of minimal degree, then \begin{equation*}\py(X) \leq \dim (\widetilde{X}) + 1.
                \end{equation*}
            \item $r+1-\qp(X)\le \py(X)$.
            \item $\qp(X)\geq \qp(X')$ if $X'$ is a variety containing $X$.
            \item $\qp(X)\geq \ell(X)$.
        \end{enumerate}
    \end{theorem}

\begin{remark}\label{rem:codim_two}
	The quadratic persistence and the Pythagoras number of totally real projective varieties of codimension 2 are completely determined by \Cref{thm:BSSV1} and \Cref{thm:BSSV2}. Indeed, in this case, there are only four possibilities:
	\begin{enumerate}
		\item $X$ is a variety of minimal degree;
		\item $X$ is an aCM variety with $d=c+2$;
		\item $X$ is contained in a unique hyperquadric;
		\item $X$ is not contained in any hyperquadric.
	\end{enumerate}
	Cases (1) and (2) are resolved by \Cref{thm:BSSV1} and \Cref{thm:BSSV2}, and case (4) is trivial. In case (3), one can easily see $\qp(X)=1$ and $\py(X)=n+2$ from \Cref{thm:qpmain}.
\end{remark}

Hence the next open case is the case of codimension three.
We consider the codimension three varieties in the next section.

\begin{example}
    Let $\mathcal{C}'\subset\p^2$ be a totally real plane curve of degree at least $7$ and $\mathcal{C} = \nu_3(\mathcal{C}')\subset\p^9$ be the curve contained in a Veronese variety $\nu_3(\p^2)\subset \p^9$. 
    Then $I(\mathcal{C})_2=I(\nu_3(\p^2))_2$.
    Also, $$\qp(\mathcal{C}) = \qp(\nu_3(\p^2)) = 6 \text{ and } \py(\mathcal{C}) = \py(\nu_3(\p^2)) = 4.$$ 
    Indeed, $\qp(\nu_3(\p^2)) = 6 \le \qp(\mathcal{C}) \le \frac{-1+\sqrt{217}}{2}=6.87...$ since $\binom{\qp(\mathcal{C})+1}{2}\le 27 = \dim_\c I(\mathcal{C})_2$ (\textit{cf.} \cite[Corollary 3.7]{MR4397034}).
    Applying the upper bounds of the Pythagoras number in \cite[Theorem 1.1]{MR4397034}, we get the following upper bounds of $\py(\mathcal{C})$:
    \begin{enumerate}[\rm (i)]
        \item $\py(\mathcal{C}) < 7$ since $\binom{\py(\mathcal{C})+1}{2} < \dim_{\c}(S_{\mathcal{C}})_2 = 28$;
        \item $\py(\mathcal{C}) \le 10=\dim\p^9+1-\min\{a(\mathcal{C}),\codim(\mathcal{C})\}$;
        \item $\py(\mathcal{C}) \le 5 \le \min_{\widetilde{X}\in A}\{1+\dim \widetilde{X}\}$ where $A$ denotes the set of totally real varieties of minimal degree containing $\mathcal{C}$.
        Indeed, $\mathcal{C}$ is not contained in any threefold variety of minimal degree since $I(\mathcal{C})_2=I(\nu_3(\p^2))_2$ and $\nu_3(\p^2)$ is not contained in any threefold variety of minimal degree, see \cite[p. 151]{MR739785}. 
    \end{enumerate}
\end{example}

\begin{remark}
Recent developments following \cite{MR4397034} include applications to nonconvex optimization \cite{blekherman2024spurious} and new results on the Pythagoras numbers of ternary forms \cite{blekherman2024pythagoras}. 
\cite{MR4397034} also provides theoretical support for computational approaches to bounding Pythagoras numbers as shown in \cite{arevalo2022minimizacion}.
An overview of recent progress and related problems is provided in the survey \cite{MR4249429}.
\end{remark}

\bigskip

\section{Codimension three varieties}\label{sec:Codim3}
Let $X \subset \mathbb{P}^r$ be a nondegenerate projective variety of codimension three and degree $d$. Then 
\begin{equation*}0 \leq \ell (X) \leq 3.
\end{equation*}
Furthermore, $\ell (X) = 3$ if and only if $d = 4$, and $\ell (X) = 2$ if and only if either $d=5$ or else $d \geq 6$ and $X$ is contained in a variety of minimal degree as a divisor (\textit{cf.} \Cref{{betti_remark}}.\eqref{ell_def}).
The main purpose of this section is to prove \Cref{mainthm:codim3} below, which shows that both $\qp(X)$ and $\py(X)$ are completely determined by $\ell (X)$. 

\begin{theorem}\label{mainthm:codim3}
Let $X \subset \p^{r}$ be a nondegenerate variety of codimension $3$ and degree $d \geq 6$.
Then $$\ell(X) = \qp(X).$$ 
In particular, $\qp(X) = 2$ if and only if $X$ is a divisor of a variety of minimal degree.
Consequently, if $X$ is a totally real variety of codimension at most three, then $$\py(X) = r + 1 - \qp(X).$$
\end{theorem}

Concerning the proof of \Cref{mainthm:codim3}, recall that if $\ell (X) = 2$ and $d \geq 6$, then $X$ is contained in a variety $Y$ of minimal degree as a divisor. In this case, if $Y$ is a totally real variety, then by Theorem 2.6 we obtain  
\begin{equation*}
\py(X) \leq \dim(X)+2 ~ \mbox{and} ~ \qp(X) \geq 2.
\end{equation*}
For this reason, ensuring that such a variety $Y$ is totally real is a crucial issue.
We address this by establishing a totally real version of the $\mathcal{K}_{p,1}$-theorem as follows.

\begin{theorem}[totally real $\mathcal{K}_{p,1}$-theorem]\label{lem:totally_real_Kp1}
    Let $X\subset\p^r_\c$ be a nondegenerate totally real variety of codimension $c$ and degree $d$. 
    Suppose either
    \begin{enumerate}[\rm (1)]
        \item\label{lem:totally_real_Kp1_1} $d\geq c+3$ and $\beta_{c-1,1}(X)\neq 0$ or
        \item\label{lem:totally_real_Kp1_2} $d=c+2$ and $X$ is a non-aCM variety that is not (a cone over) the isomorphic projection of the Veronese surface in $\p^5$.
    \end{enumerate}
    Then there exists a totally real variety of minimal degree containing $X$ as a divisor.
\end{theorem}
\begin{proof}
    First, note that in case \eqref{lem:totally_real_Kp1_1}, $X$ is a divisor of a variety of minimal degree by Green's $\mathcal{K}_{p,1}$-theorem (\textit{cf.} \Cref{betti_remark}\eqref{ell_def}).
    In case \eqref{lem:totally_real_Kp1_2}, the same holds by \cite{MR2274517}.
    Let $Y\subset\p^r$ be a variety of minimal degree containing $X$ as a divisor.
    We claim that $Y$ is totally real.
    As $X$ is totally real, there is an associated ideal $I_\r(X)\subset\r[x_0,\cdots,x_r]$ (\textit{cf.} \Cref{rem:tr}).

    Indeed, let $V_\r=\r\cdot x_0+\cdots+\r\cdot x_r$ and $V_\c=\c\cdot x_0+\cdots+\c\cdot x_r$.
    Thus $V_\r\subset V_\c$.
    Then $\beta_{c-1,1}(X) \not= 0$ in both cases \eqref{lem:totally_real_Kp1_1} and \eqref{lem:totally_real_Kp1_2}, and hence the Koszul cohomology group $K_{c-2,2}(I_\r(X),V_\r)$ is nonzero. 
    Now, pick a nonzero element $\gamma$ of $K_{c-2,2}(I_\r(X),V_\r)$.
    Thus it is contained in the kernel of
    \[
        \partial_\r:\wedge^{c-2}V_\r\otimes I_\r(X)_2\to \wedge^{c-3}V_\r\otimes I_\r(X)_3.
    \]
    From the commutative diagram
    \[\begin{tikzcd}
        \wedge^{c-2}V_\r\otimes I_\r(X)_2\arrow[r,"\partial_\r"]\arrow[d,"i"] &\wedge^{c-3}V_\r\otimes I_\r(X)_3\arrow[d,"j"]\\
        \wedge^{c-2}V_\c\otimes I(X)_2\arrow[r,"\partial_\c"] &\wedge^{c-3}V_\c\otimes I(X)_3
    \end{tikzcd}\]
    of $\r$-vector spaces, one can see that $i(\gamma)\in \ker \partial_\c=K_{c-2,2}(I(X),V_\c)$.
    For both cases \eqref{lem:totally_real_Kp1_1} and \eqref{lem:totally_real_Kp1_2}, we have $\beta_{c-1,1}(X)=c-1$.
    Hence $\beta_{c-1,1}(X)=\beta_{c-1,1}(Y)$, so $$K_{c-2,2}(I(X),V_\c)=K_{c-2,2}(I(Y),V_\c).$$
    Thus $i(\gamma)\in K_{c-2,2}(I(Y),V_\c)$, so the syzygy ideal of $i(\gamma)$ is $I(Y)$.
    By construction, the syzygy ideal of $i(\gamma)$ is the extension of the syzygy ideal of $\gamma\in K_{c-2,2}(I_\r(X),V_\r)$ under the map $\r[x_0,\cdots,x_r]\to \c[x_0,\cdots,x_r]$.
    This shows $I(Y)$ can be generated by real polynomials.
    As the singular locus of $Y$ is degenerate and $Y$ contains $X(\r)$, which is nondegenerate, $Y$ contains a smooth real point and is totally real.
\end{proof}

Next, in order to investigate the relationship between $\ell (X)$ and $\qp(X)$ mentioned in \Cref{mainthm:codim3}, we first prove the following proposition.

\begin{proposition}\label{prop:qp(X) bigger than 2}
Let $X \subset \p^r$ be a nondegenerate projective variety of codimension $\geq 3$. Then $\qp (X) \geq 2$ if and only if $\beta_{2,1} (X) > 0$.
\end{proposition}

To prove this, we first show the following lemma.

\begin{lemma}\label{lem:degeneracy singular locus}
Let $Y \subset \p^r$ $(r \geq 3)$ be a complete intersection of two quadrics. 
If $Y$ is irreducible, then every irreducible component of $\Sing (Y)$ is degenerate in $\p^r$.
\end{lemma}

\begin{proof}
Let $R = \mathbb{K} [x_0 , x_1 , \ldots , x_r ]$ be the homogeneous coordinate ring of $\p^r$. If $Y$ is smooth, we are done.

Next, suppose that $Y$ is singular but it has no vertex point. 
Let $p$ be a singular point of $Y$ and let $Y_p \subset \p^{r-1}$ be the inner projection of $Y$ from $p$. 
Then $Y_p$ is a quadratic hypersurface. Let $Q_1$ be the defining equation of $Y_p$ and let $Q_2$ be the quadratic form such that $I(Y) = \langle Q_1 , Q_2 \rangle$. 
We may assume that $p = [1:0:\cdots:0]$. 
Then $Q_1$ is contained in the subring $\mathbb{K} [x_1 , \ldots , x_r ]$ of $R$ and hence $\partial{Q_1}/\partial{x_0} = 0$. Then $\partial{Q_1}/\partial{x_i} \neq 0$ for some $1 \leq i \leq r$. 
Since $p$ is not a vertex point of $Y$, it follows that $\partial{Q_2}/\partial{x_0} \neq 0$. 
Thus from the Jacobian matrix of $Y$, the homogeneous ideal of $\mbox{Sing} (Y)$ contains the nonzero homogeneous quadratic polynomial
$$\frac{\partial{Q_1}}{\partial{x_i}} \times \frac{\partial{Q_2}}{\partial{x_0}} .$$
This shows that every irreducible component of $\mbox{Sing} (Y)$ is degenerate in $\p^r$.

Finally, suppose that $Y$ is a cone over $Z \subset \p^{r-s}$ where $Z$ is not a cone. Then
$$\mbox{Vert}(Y) = \p^{s-1} \quad \mbox{for some}\ s \geq 1$$
and $\mbox{Sing} (Y)$ is the join of $\mbox{Vert}(Y)$ and $\mbox{Sing} (Z)$. Also every irreducible component of $\mbox{Sing} (Z)$ is degenerate in $\p^{r-s}$. Therefore every irreducible component of $\mbox{Sing} (Y)$ is degenerate in $\p^r$.
\end{proof}

\begin{proof}[Proof of \Cref{prop:qp(X) bigger than 2}]
Suppose that $\qp (X) \geq 2$ and let $X_p \subset \p^{r-1}$ be the inner projection of $X$ from a general point $p \in X$. Since $\qp(C) \geq 2$, it holds that $h^0 (\p^r,\i_X (2)) \geq 2$. Suppose that there exist $\mathbb{K}$-linearly independent quadrics $Q_1$ and $Q_2$ in $\H^0 (\p^r,\i_X (2))$ such that $Y = Q_1 \cap Q_2$ is not irreducible. Then, by Bezout's theorem, the minimal irreducible decomposition of $Y$ should be
$$Y = Y_1 \cup Y_2$$
where $Y_1$ and $Y_2$ are $(r-2)$-dimensional such that $\mbox{deg} (Y_1 ) =1$ and $\mbox{deg} (Y_2 ) =3$. Then $Y_1 = \p^{r-2}$ and $Y_2$ is a variety of minimal degree. Therefore $X$ should be contained in $Y_2$ since it is nondegenerate in $\p^r$. Then we have
$$\beta_{2,1} (X) \geq \beta_{2,1} (Y_2) >0. $$
Now, we consider the case where $Y = Q_1 \cap Q_2$ is irreducible for any two $\mathbb{K}$-linearly independent quadrics $Q_1$ and $Q_2$ in $\H^0 (\p^r,\i_X (2))$. By Lemma \ref{lem:degeneracy singular locus}, every irreducible component of the singular locus of $Y$ is degenerate in $\p^r$ and hence a general point $p$ of $X$ is in the smooth locus of $Y$. Then $\pi_p (Y) \subset \p^{r-1}$ is a cubic hypersurface. In particular, $X_p$ satisfies no nonzero quadratic polynomials of the form $aQ_1 +bQ_2$. This implies that $X_p$ satisfies no quadratic equation and hence $\qp (X) = 1$. This is a contradiction.

If $\beta_{2,1} (X) > 0$, then $\qp(X) \geq 2$ since the general inner projection of $X$ is contained in a quadric hypersurface.
\end{proof}

We are now in a position to prove \Cref{mainthm:codim3}.

\begin{proof}[Proof of \Cref{mainthm:codim3}]
We have $\qp(X)\leq 2$ due to \Cref{thm:BSSV1}.
By \Cref{prop:qp(X) bigger than 2}, we get $\qp(X)=\ell(X)$.
If $X$ is totally real, then $\py(X)\geq r+1-\qp(X)$ by \Cref{thm:qpmain}.
When $\qp(X)=2$, we have $\ell(X)=2$ so that $X$ is a divisor of a totally real variety $Y$ of minimal degree by \Cref{lem:totally_real_Kp1}.
Hence $\py(X)\leq r-1$ and we get $\py(X)=r-1$.
When $\qp(X)=1$, let $Q$ be the quadric hypersurface containing $X$.
As $X$ is totally real, we may choose $Q$ to be defined by a real quadratic form. 
As the singular locus of $Q$ is degenerate and $Q$ contains a nondegenerate set $X(\r)$ of real points, the hypersurface $Q$ contains a smooth real point. 
Hence $Q$ is totally real.
Then $\py(X)\leq \py(Q)=r$, and we get $\py(X)=r$.
When $\qp(X)=0$, we immediately get $\py(X)=\py(\p^r)=r+1$.
\end{proof}

\bigskip

\section{The main theorem}\label{sec:mainthm}
In this section, we prove the main theorem (\Cref{mainthm}).
To prove it, we need to solve the lifting problem for $d\geq 2c+3$.
We start with reminding the definition of the partial elimination ideals.
Let $S=k[x_0,\cdots,x_r]$ and pick a point $q\in \p^r$.
By a coordinate change, we may assume $q=[1,0,\cdots,0]$.
Denote $R=k[x_1,\cdots,x_r]$.
For a homogeneous ideal $I\subset S$, we define the $i$-th partial elimination ideal of $I$ with respect to $q$ as follows:
\[
    K_i(I,q):=\{f\in R\mid \exists F\in I \textit{ s.t. }F=fx_0^i+(\text{lower $x_0$-degree terms})\}\cup \{0\}.
\]
Then $K_i(I,q)$ is a homogeneous ideal of $R$. For the detailed theory regarding partial elimination ideals, \textit{cf.} \cite{MR1648665}.

On the other hand, the result of Castelnuovo \cite{castelnuovo1889ricerche} implies the following:
\begin{itemize}
    \item when $d\geq 2c+1$, $\dim I(X)_2\leq \binom{c}{2}$;
    \item when $d\geq 2c+3$, if $\dim I(X)_2=\binom{c}{2}$, then $X$ is contained in a variety of minimal degree as a divisor.
\end{itemize}

Using these, we solve the lifting problem for $d\geq 2c+3$.

\begin{theorem}\label{lp_2c+3}
    Let $X\subset\p^r$ be a nondegenerate projective variety of dimension $n$, codimension $c$, and degree $d\geq 2c+3$. 
    Let $X_q\subset\p^{r-1}$ denote the inner projection of $X$ from a general point $q\in X$.
    If $X_q$ is contained in a variety of minimal degree as a divisor, then so is $X$.
\end{theorem}
\begin{proof}
    Denote $p_0:=q$ and $Y_1:=Y$.
    Let $p_1$ be a general point of $Y_1$.
    Denote the projection of $Y_1\subset\p^{r-1}$ from $p_1$ by $Y_2\subset\p^{r-2}$, and let $p_2$ be a general point of $Y_2$.
    By repeating this, we get $p_0,p_1,\cdots,p_{c-2}$ and $Y_{c-1}=\p^{n+1}$.
    We pick $n+2$ general points in $Y_{c-1}=\p^{n+1}$ and denote them by $p_{c-1},p_c,\cdots,p_r$.
    Then $p_0,p_1,\cdots,p_r$ are linearly independent.
    Let $\{x_i = p_i^* \mid 0 \leq i \leq r \}$ be the dual basis of $\{p_0,p_1,\cdots,p_r\}$.
    Thus $x_0,x_1,\cdots, x_r$ form a homogeneous coordinate system of $\p^r$.
    Denote by $S$ and $R$ the homogeneous coordinate rings of $\p^r$ and $\p^{r-1}$, respectively.
    Let $J\subset S$ be the ideal generated by $I(X)_2$ and $W\subset\p^r$ be the scheme defined by $J$.
    Then $$(J\cap R)_2=(I(X)\cap R)_2=I(X_q)_2.$$
    As $d\geq 2c+3$, we have $\deg X_q\geq 2\codim X_q+4$. 
    Hence $\dim I(X_q)_2\leq \binom{\codim X_q}{2}$. 
    However, as $X_q$ is contained in a variety of minimal degree as a divisor, it holds that 
    $$\dim I(X_q)_2=\binom{\codim X_q}{2}\text{, and therefore}~ I(X_q)_2=I(Y)_2.$$
    In summary, we have $(J\cap R)_2=I(Y)_2$ and $J\cap R\supset I(Y)$.
    
    As $p_0$ is not contained in the vertex set $\ver(W_{\text{red}})$ of $W_{\text{red}}$, there exists $f_0\in J_2$ such that $\frac{\partial f_0}{\partial x_0}\neq 0$.
    Also as $p_1\notin\ver(Y)\subset\p^{r-1}$, there exists $f_1\in I(Y)_2$ such that $\frac{\partial f_1}{\partial x_1}\neq 0$. Since $f_1\in R$, we have $\frac{\partial f_1}{\partial x_0}=0$.
    Let $Y_{p_1}\subset\p^{r-2}$ be the inner projection of $Y$ from $p_1$.
    As $p_2\notin\ver(Y_{p_1})$, there exists $f_2\in I(Y_{p_1})_2$ such that $\frac{\partial f_2}{\partial x_2}\neq 0$ while $\frac{\partial f_2}{\partial x_0}=\frac{\partial f_2}{\partial x_1}=0$.
    By repeating this, we get $f_0,f_1,\cdots,f_{c-2}\in J_2$ such that $\frac{\partial f_i}{\partial x_i}\neq 0$ while $\frac{\partial f_i}{\partial x_j}=0$ for all $j<i$.
    Hence by taking a proper generating set of $J$, one can regard the Jacobian matrix of $W\subset\p^r$ contains the following form of $(c-1) \times (c-1)$ submatrix where $*$ represents a nonzero linear form.
    \[
    M=
    \begin{bmatrix}
        * & 0 & 0 & 0 & 0 \\
        ? & * & 0 & 0 & 0 \\
        ? & ? & * & 0 & 0 \\
        ? & ? & ? & * & 0 \\
        ? & ? & ? & ? & * \\
    \end{bmatrix}
    \]
    Note that $\dim W = n$ or $n+1$. Define the locus $U:=\{x\in W_{\text{red}}\mid \dim T_xW\geq n+2\}$.
    Then $U\subset V(\det M)$ as sets.
    As $\det M$ is decomposable to linear forms, every irreducible component of $U$ is degenerate.
    Note that the scheme $W$ and the locus $U$ does not depend on the choice of $q\in X$.
    Therefore, $q\notin U$ and $\dim T_qW\leq n+1$.
    Hence $$\dim (K_1(I(X),q))_1=\dim (K_1(J,q))_1=\codim T_qW\geq c-1.$$
    Since $\dim I(X)_2=\dim I(X_q)_2+\dim (K_1(I(X),q))_1$, we obtain
    $$\dim I(X)_2\geq\binom{e-1}{2}+c-1=\binom{c}{2}.$$
    Since we have $\dim I(X)_2\leq \binom{c}{2}$ from the condition, $$\dim I(X)_2=\binom{c}{2}.$$
    Together with $d\geq 2c+3$, we conclude that $X$ is contained in a variety of minimal degree as a divisor.
\end{proof}

By combining the totally real $\mathcal{K}_{p,1}$-theorem (\Cref{lem:totally_real_Kp1}), the result for codimension three varieties (\Cref{mainthm:codim3}), and the lifting theorem (\Cref{lp_2c+3}), we prove the main theorem.

\begin{proof}[Proof of \Cref{mainthm}]
    To show \eqref{mainthm:VAMD}, assume $d=c+2$ and $X$ is not aCM. Then $X$ is contained in a variety $Y$ of minimal degree as a divisor, so $$\qp(X)\geq \qp(Y)=c-1.$$ 
    As $\qp(X)\lneq c$, we get $\qp(X)=c-1$.
    Thus $$\py(X)\geq r+1-\qp(X)=n+2.$$
    By the totally real $\mathcal{K}_{p,1}$-theorem (\Cref{lem:totally_real_Kp1}), $Y$ must be totally real.
    Hence $$\py(X)\leq \py(Y)=n+2.$$

    To show \eqref{mainthm:general}, first note that the case $c=3$ is solved by \Cref{mainthm:codim3}.
    Assume $c\geq 4$ and $d\geq 2c+3$.
    The directions (iii)$\Rightarrow$(ii)$\Rightarrow$(i) follow from the totally real $\mathcal{K}_{p,1}$-theorem (\Cref{lem:totally_real_Kp1}) and the results of \cite{MR4397034}.
    To show (i)$\Rightarrow$(iii), assume $\qp(X)=c-1$. By taking general inner projections of $X$, we get a variety $\overline{X}\subset \p^{n+3}$.
    Then $\qp(\overline{X})=2$.
    By \Cref{mainthm:codim3}, $\overline{X}$ is a divisor of a variety of minimal degree.
    By applying \Cref{lp_2c+3} inductively, we conclude that $X$ must be a divisor on a variety of minimal degree.
\end{proof}

\begin{remark}\label{transform_to_lp}
	If the statement of the lifting problem holds for all $d\geq c+3$, then the aCM condition in \Cref{thm:BSSV2} can be removed. More precisely, as in the proof above, it suffices to show (i)$\Rightarrow$(iii) in the case $d\geq c+3$. If $X$ satisfies $\qp(X)=c-1$, then by taking successive general inner projections, we obtain a codimension three variety $Y$ with $\qp(Y)=2$. Therefore, by \Cref{mainthm:codim3}, $Y$ is a divisor of a variety of minimal degree. The lifting problem asserts that $X$ is also a divisor of a variety of minimal degree, which allows us to conclude (iii).
\end{remark}

\bigskip

\section{Curves of maximal regularity}\label{sec:MaxRegCurves}
In this section, we study the quadratic persistence and the Pythagoras number of curves of maximal regularity. 
Let $\mathcal{C}\subset\p^r$ a nondegenerate curve of maximal regularity of degree $d$. Note that when $r\leq 3$, the quadratic persistence and the Pythagoras numbers are easily determined, so we assume $r\geq 4$.
\begin{remark}
If the curve $\mathcal{C}\subset\p^r$ has the maximal regularity, then there exists an extremal secant line $L$ to $\mathcal{C}$, which allows the construction of a rational normal threefold scroll containing $\mathcal{C}$.  
More precisely, $\mathcal{C}$ is a smooth rational curve and has a unique $(d - r + 2)$-secant line $L$ \cite{MR0704401}.  
Moreover, the image of $\mathcal{C}$ under the projection from $L$ has degree $r - 2$ in $\mathbb{P}^{r-2}$, i.e., it is the rational normal curve $S(r-2)$.
This leads to construct a rational normal threefold scroll $J(L,S(r-2))=S(0,0,r-2)$ containing $\mathcal{C}$ where $J(X,Y)$ denotes the join of $X$ and $Y$.
Note that if $\mathcal{C}$ is contained in $S(1, r-2)$, then the extremal secant line is $S(1) \subset S(1,r-2)$.  
\end{remark}

Unlike the usual situation, when we consider whether the image of a projection of a complex projective variety is totally real, the linear subspace corresponding to the codomain of the projection matters.

\begin{proposition}\label{prop:tr_projection}
    Let $X\subset\p^r$ be a nondegenerate totally real variety. If $L,\Lambda\subset\p^r$ are totally real linear subspaces with $\dim L+\dim\Lambda=r-1$, then the image $\overline{\pi_L(X\backslash L)}$ of $X$ under the projection $\pi_L:\p^r\backslash L\to \Lambda$ is totally real.
\end{proposition}
\begin{proof}
    As $L$ and $\Lambda$ are totally real, these can be defined by real linear forms.
    Then the result follows since for any $x\in X(\r)$, the point $\Lambda\cap \langle L,x\rangle$ is real.
\end{proof}

In the following lemma, we solve the lifting problem for curves of maximal regularity.

\begin{lemma}\label{lem:maxreg_lifting}
    Let $\mathcal{C}\subset\p^r$ ($r\geq 5$) be a nondegenerate curve of maximal regularity of degree $d\geq r+2$ and $L$ be its $(d-r+2)$-secant line. Let $p\in \mathcal{C}$ be the general point and $\mathcal{C}_p\subset\p^{r-1}$ be the inner projection of $\mathcal{C}$ from $p$. If $\mathcal{C}_p$ is contained in $S(1,r-3)$, then $\mathcal{C}$ is contained in $S(1,r-2)$.
\end{lemma}
\begin{proof}
    Denote by $\Lambda_{r-1} := \mathbb{P}^{r-1}$ the hyperplane on which $\mathcal{C}_p$ lies.
    The curve $\mathcal{C}_p$ is also a curve of maximal regularity with the $(d-r+2)$-secant line $L_p$ which is a projection of $L$ from $p$.
    Hence $\mathcal{C}_p$ is a smooth rational curve, and the inner projection of $\mathcal{C}$ from $p$ is an isomorphic projection.
    Suppose $\mathcal{C}_p\subset S(1,r-3)$.
    Then $L_p$ corresponds to $S(1)\subset S(1,r-3)$.
    Pick $S(r-3)\subset S(1,r-3)$.
    As its linear span $\Lambda_{r-3}:=\p^{r-3}$ does not intersect with $L_p=S(1)\subset S(1,r-3)$, it does not intersect $\langle p,L\rangle$.
    Then $\Lambda_{r-3}$, $L$, $p$ are linearly independent, i.e., they span the whole $\p^r$.
    Define $\Lambda_{r-2}:=\langle p,\Lambda_{r-3}\rangle\cong\p^{r-2}$. Then $\Lambda_{r-2}$ does not intersect $L$.
    To summarize, we have the following commutative diagram.
    \[\begin{tikzcd}
        \p^r\arrow[r,dashed,"\pi_p"]\arrow[d,dashed,"\pi_L"]&\Lambda_{r-1}\arrow[d,dashed,"\pi_{L_p}"]\\
        \Lambda_{r-2}=\langle p,\Lambda_{r-3}\rangle\arrow[r,dashed,"\pi_p"]&\Lambda_{r-3}
    \end{tikzcd}\]
    By taking the projection of $\mathcal{C}$ to $\Lambda_{r-2}$ from the center $L$, we get the rational normal curve $S(r-2)\subset \Lambda_{r-2}$ that contains $p$. 
    Since $S(r-3)\subset\Lambda_{r-3}$ is the projection of $\mathcal{C}_p\subset\Lambda_{r-1}$ from $L_p$, it is the projection of $\mathcal{C}$ from $\langle p,L\rangle$.
    Hence $S(r-3)\subset\Lambda_{r-3}$ is the projection of $S(r-2)\subset \Lambda_{r-2}$ from $p$ since the former one is $\pi_{\langle p,L\rangle}(\mathcal{C})$ and the latter one is $\pi_L(\mathcal{C})$.
    Therefore, we obtain the following commutative diagram.
    \[\begin{tikzcd}
        \mathcal{C}\arrow[r,dashed,"\pi_p"]\arrow[d,dashed,"\pi_L"]&\mathcal{C}_p\arrow[d,dashed,"\pi_{L_p}"]\\
        S(r-2)\arrow[r,dashed,"\pi_p"]&S(r-3)
    \end{tikzcd}\]
    Note that the rational maps $\pi_p|_{\mathcal{C}}:\mathcal{C}\dashrightarrow \mathcal{C}_p$ and $\pi_p|_{S(r-2)}:S(r-2)\dashrightarrow S(r-3)$ can be extended to isomorphisms $\pi_p|_{\mathcal{C}}:\mathcal{C}\to \mathcal{C}_p$ and $\pi_p|_{S(r-2)}:S(r-2)\to S(r-3)$.
    
    For a point $t\in \p^1$, we have a ruling $R(t)\cong \p^1\subset S(1,r-3)\subset\Lambda_{r-1}$. Denote its intersections with $L_p$, $S(r-3)$, and $\mathcal{C}_p$ as \[l(t)\in L_p,\ s(t)\in S(r-3),\ \text{and }c(t)\in \mathcal{C}_p,\] respectively.
    As $L_p$, $S(r-3)$, and $\mathcal{C}_p$ are projections of $L$, $S(r-2)$, and $\mathcal{C}$ from $p$, we get \[\widetilde{l}(t)\in L,\ \widetilde{s}(t)\in S(r-2),\ \text{and }\widetilde{c}(t)\in \mathcal{C}\] that are projected to $l(t)$, $s(t)$, and $c(t)$, respectively. Note that such liftings are unique as the projections of $L$, $S(r-3)$, and $\mathcal{C}$ from $p$ are all isomorphisms.
    To show the ruling can be lifted, it is enough to show $\widetilde{l}(t)$, $\widetilde{s}(t)$, $\widetilde{c}(t)$ are collinear.
    Suppose not. Note that $L$ can not be contained in the 2-plane $\langle p,\widetilde{l}(t),\widetilde{s}(t),\widetilde{c}(t),l(t),s(t),c(t)\rangle$ since if it is, then $L_p=S(1)\subset S(1,r-3)$ intersects with $S(r-3)\subset S(1,r-3)$ at $s(t)$ which is a contradiction.
    Hence the assumption that $\widetilde{l}(t)$, $\widetilde{s}(t)$, $\widetilde{c}(t)$ are not collinear implies $\langle \widetilde{s}(t),L\rangle \neq \langle \widetilde{c}(t),L\rangle$. As $\widetilde{c}(t)$ is projected to $c(t)$ by $\pi_p$ and $c(t)$ is projected to $s(t)$ by $\pi_{L_p}$, the point $\widetilde{c}(t)$ is projected to $s(t)$ by $\pi_{\langle p,L\rangle}$.
    \[\begin{tikzcd}[sep=large]
        \widetilde{c}(t)\arrow[r,"\pi_p",maps to]\arrow[rd,"\pi_{\langle p,L\rangle}",maps to]&c(t)\arrow[d,"\pi_{L_p}",maps to]\\
        \widetilde{s}(t)\arrow[r,"\pi_p",maps to]&s(t)
    \end{tikzcd}\]
    As $\widetilde{s}(t)$ is the unique point of $S(r-2)$ projected to $s(t)$ by $\pi_p$, the point $\widetilde{c}(t)$ must be projected to $\widetilde{s}(t)$ by $\pi_L$. Hence $\langle \widetilde{s}(t),L\rangle=\langle \widetilde{c}(t),L\rangle$, and we get a contradiction. Therefore $\widetilde{l}(t)$, $\widetilde{s}(t)$, $\widetilde{c}(t)$ are collinear.
\end{proof}

As a consequence, we get \Cref{mainthm:cmr}.

\begin{proof}[Proof of \Cref{mainthm:cmr}]
    First, consider the case $r=4$.
    As $\mathcal{C}$ is contained in $S(0,0,2)$, we obviously have $\qp(\mathcal{C})\geq 1$, so $\qp(\mathcal{C})=1$ or $2$.
    The other statements of this case follows from \Cref{mainthm:codim3}.
    Note that if a surface $S\subset\p^4$ of minimal degree contains $\mathcal{C}$, then it should be a smooth rational normal scroll as $\mathcal{C}$ is not aCM, so $S=S(1,2)$.
    
    Now suppose $r\geq 5$.
    We show the equivalence first.
    As before, the directions (iii)$\Rightarrow$(ii)$\Rightarrow$(i) follows from the totally real $\mathcal{K}_{p,1}$-theorem and the results of \cite{MR4397034}.
    To show (i)$\Rightarrow$(iii), assume $\qp(\mathcal{C})=r-2$.
    By taking general inner projections of $\mathcal{C}\subset\p^r$, we get a curve $\overline{\mathcal{C}}\subset\p^4$ of maximal regularity. Note that $\qp(\mathcal{C})=\qp(\overline{\mathcal{C}})+r-4$.
    Hence $\qp(\overline{\mathcal{C}})=2$, so $\overline{\mathcal{C}}$ is contained in $S(1,2)$. By using \Cref{lem:maxreg_lifting} inductively, we conclude $\mathcal{C}\subset S(1,r-2)$. This shows the equivalence.
    
    The part $\qp(\mathcal{C})\in\{r-3,r-2\}$ follows from $\qp(\overline{\mathcal{C}})\in\{1,2\}$ where $\overline{\mathcal{C}}\subset\p^4$ is defined as above.
    
    Then we show $\py(\mathcal{C})=r+1-\qp(\mathcal{C})$.
    As the case $\qp(\mathcal{C})=r-2$ follows from the above, we assume $\qp(\mathcal{C})=r-3$ and show $\py(\mathcal{C})=4$.
    Note that in the proof of \Cref{lem:maxreg_lifting}, we may assume $p$ is a real point and both $\Lambda_{r-1}$ and $\Lambda_{r-3}$ are totally real. Then $\Lambda_{r-2}=\langle p,\Lambda_{r-3}\rangle$ is also totally real.
    Recall that the projection of $\mathcal{C}$ from its maximal secant line $L$ is $S(r-2)\subset\p^{r-2}$.
    As $\mathcal{C}$ is totally real, $I(\mathcal{C})$ can be generated by real homogeneous polynomials. Fix a minimal generating set $A$ of $I(\mathcal{C})$ consists of real forms. 
    Let $f$ be the unique form of degree $d-r+2$ in $A$ and $J$ be the ideal generated by $A\backslash \{f\}$.
    Then the ideal of $L$ is $J:(f)$ (\textit{cf.} \cite{MR1992539}).
    Hence $L$ is totally real. As $S(r-2)$ is the projection of $\mathcal{C}$ from $L$, it is also totally real by \Cref{prop:tr_projection}. Thus the join of $L$ and $S(r-2)$ is totally real, i.e., $\mathcal{C}$ is contained in a totally real rational normal threefold scroll $S(0,0,r-2)$. Therefore $\py(\mathcal{C})\leq \py(S(0,0,r-2))=4$. Together with $\py(\mathcal{C})\geq r+1-\qp(\mathcal{C})=4$, we conclude $\py(\mathcal{C})=4$.
\end{proof}

\bigskip

\section{Smooth nonhyperelliptic curve of genus 3}\label{sec:nonhyper}

Let $\mathcal{C}$ be a smooth nonhyperelliptic curve of arithmetic  genus $3$ and $\mathcal{L}$ a very ample line bundle on $\mathcal{C}$ of degree $d \geq 6$. Also, let
\begin{equation*}
\mathcal{C} \subset \p^{r}=\p\H^0(\mathcal{C},\l)
\end{equation*}
be the linearly normal embedding of $\mathcal{C}$ by $\mathcal{L}$ where $r=d-3$. For the convenience, we denote by $I(\mathcal{C},\mathcal{L})$ the homogeneous ideal of $\mathcal{C}$ and $\qp (\mathcal{C},\mathcal{L})$ the quadratic persistence of $\mathcal{C} \subset \p^{r}$.

\begin{remark}\label{rem:Homma's work}
In \cite{MR623440}, M. Homma study the projective normality and defining equations of $\mathcal{C} \subset \p^{r}$. Here we review some of her results.
\begin{enumerate}[\rm (1)]
    \item \label{rem:Homma1} When $d = 6$, it holds that $\mathcal{L} \neq \omega_{\mathcal{C}} \otimes \mathcal{O}_{\mathcal{C}} (x+y)$ for any $x,y \in \mathcal{C}$. In this case, $\mathcal{C} \subset \p^3$ is projectively normal and $I(\mathcal{C},\mathcal{L})$ is minimally generated by $4$ cubic polynomials. For details, see Proposition 1.9 and Theorem 2.1 in \cite{MR623440}.
    \item \label{rem:Homma2} When $d=7$, we can write $\mathcal{L} = \omega_{\mathcal{C}} \otimes \mathcal{O}_{\mathcal{C}} (x+y+z)$ for some $x,y,z \in \mathcal{C}$. In this case, $\mathcal{C} \subset \p^4$ is projectively normal and $I(\mathcal{C},\mathcal{L})$ is minimally generated by $3$ quadratic polynomials and $4$ cubic polynomials.
    \item \label{rem:Homma3} When $d \geq 8$, it holds that $\mathcal{C} \subset \p^{r}$ is projectively normal and $I(\mathcal{C},\mathcal{L})$ is minimally generated by quadratic polynomials.
    \item \label{rem:Homma4} When $\mathcal{L} = \omega_{\mathcal{C}} ^2$, it holds that $\mathcal{C} \subset \p^5$ is the second Veronese embedding of the plane quartic model of $\mathcal{C}$. In particular, $\mathcal{C}$ in $\p^5$ is contained in the Veronese surface. 
\end{enumerate}  
\end{remark}
\begin{proof}[Proof of \Cref{thm:nonhyperelliptic curve of genus 3}]
First, we claim (i), (iii), and (iv) are equivalent.
As before, we have $\qp(\mathcal{C},\mathcal{L}) \leq r-2$ by \Cref{thm:BSSV1}.

When $d = 6$, it follows by \Cref{rem:Homma's work}.\eqref{rem:Homma1} that $\qp (\mathcal{C},\mathcal{L}) = 0$, hence $\py(\mathcal{C})=4$.

When $d=7$, we have $\dim_{\mathbb{C}} I(\mathcal{C},\mathcal{L})_2 = 3$ and so $\qp (\mathcal{C},\mathcal{L}) \geq 1$. Thus $\qp (\mathcal{C},\mathcal{L})$ is $1$ or $2$. Now, we will show that 
$$\qp (\mathcal{C},\mathcal{L}) = 2 \  \mbox{if and only if} \  h^0 (\mathcal{C},\mathcal{L} \otimes \omega_{\mathcal{C}} ^{-1}) =2.$$ 
Indeed, $h^0 (\mathcal{C},\mathcal{L} \otimes \omega_{\mathcal{C}} ^{-1}) \in \{1,2\}$. Also, if $h^0 (\mathcal{C},\mathcal{L} \otimes \omega_{\mathcal{C}} ^{-1}) =2$, then the multiplication map
\begin{equation*}\label{equ:multiplication}
\mu : \H^0 (\mathcal{C},\mathcal{L} \otimes \omega_{\mathcal{C}} ^{-1}) \times \H^0 (\mathcal{C}, \omega_{\mathcal{C}}) \rightarrow \H^0 (\mathcal{C},\mathcal{L})
\end{equation*}
defines a rational normal surface scroll containing $\mathcal{C}$. Hence $\qp (\mathcal{C},\mathcal{L}) = 2$. Now, suppose that $h^0 (\mathcal{C},\mathcal{L} \otimes \omega_{\mathcal{C}} ^{-1}) = 1$. Let $x \in \mathcal{C}$ be a general point. Then one can check that 
$$\mathcal{L} \otimes \mathcal{O}_{\mathcal{C}} (-x) \neq \omega_{\mathcal{C}} \otimes \mathcal{O}_{\mathcal{C}} (y+z) \ \mbox{for any} \ y,z \in \mathcal{C}.$$
Therefore the inner projection of $\mathcal{C} \subset \p^4$ from $x$ satisfies no quadratic equation. This proves the desired equality $\qp (\mathcal{C},\mathcal{L}) = 1$.

When $d=8$ and $\mathcal{L} \neq \omega_{\mathcal{C}} ^2$, we write $\mathcal{L} = \omega_{\mathcal{C}} \otimes \mathcal{M}$ where $\mathcal{M}$ is a non-special line bundle of degree $4$. Then $h^0 (\mathcal{C},\mathcal{M})=2$ and hence 
\begin{equation*}
h^0 (\mathcal{C},\mathcal{L} \otimes \mathcal{O}_{\mathcal{C}} (-x) \otimes \omega_{\mathcal{C}} ^{-1} ) = h^0 (\mathcal{C},\mathcal{M} \otimes \mathcal{O}_{\mathcal{C}} (-x)) = 1
\end{equation*}
if $x \in \mathcal{C}$ is not a base point of $\mathcal{M}$. 
This implies that $\qp (\mathcal{C},\mathcal{L} \otimes \mathcal{O}_{\mathcal{C}} (-x) ) = 1$ for general $x \in \mathcal{C}$ and hence $\qp (\mathcal{C},\mathcal{L}) = 2$. 
When $\mathcal{L} = \omega_{\mathcal{C}} ^2$, Remark \ref{rem:Homma's work}.(4) says that $\mathcal{C}$ in $\p^5$ is contained in the Veronese surface. 
Hence $\qp (\mathcal{C},\omega_{\mathcal{C}} ^2 ) = 3$.

When $d = 9$, it holds that $\mathcal{L} \otimes \mathcal{O}_{\mathcal{C}} (-x) \neq \omega_{\mathcal{C}}^2$ for general $x \in \mathcal{C}$. 
Therefore, if $x$ is a general point of $\mathcal{C}$, then $\qp (\mathcal{C},\mathcal{L} \otimes \mathcal{O}_{\mathcal{C}} (-x) ) = 2$. 
Thus $\qp (\mathcal{C},\mathcal{L}) = 3$.

When $d \geq 10$, choose general points $x_1, \ldots, x_{d-9} \in \mathcal{C}$. 
Then $\mathcal{O}_{\mathcal{C}} (-x_1 - \cdots - x_{d-9})$ is of degree $9$ and hence $\qp (\mathcal{C},\mathcal{L} \otimes \mathcal{O}_{\mathcal{C}} (-x_1 - \cdots - x_{d-9}) ) = 3$. Hence $\qp (\mathcal{C},\mathcal{L}) = d-6$. 
This shows the claim.

Finally, the directions (iii)$\Rightarrow$(ii)$\Rightarrow$(i) can be obtained as before.
\end{proof}

\begin{remark}
    When $\qp(\mathcal{C}) = r - 3$, the inequality \eqref{equ:py,qp} provides a lower bound $\py(\mathcal{C}) \geq 4$. 
    On the other hand, an upper bound derived from the Green-Lazarsfeld index (\cite[Theorem 1.1.\rm {(ii)}]{MR4397034}) gives $\py(\mathcal{C}) \leq 5$, thus restricting $\py(\mathcal{C})$ to either $4$ or $5$.
    One can see there exists a threefold $Y$ of minimal degree containing $\mathcal{C}$, but we do not know whether $Y$ is totally real.
    If $Y$ is totally real, then it concludes that $\qp(\mathcal{C}) = r - 3$ if and only if $\py(\mathcal{C}) = 4$ so that $\py(\mathcal{C})=r+1-\qp(\mathcal{C})$.
\end{remark}

\bibliographystyle{amsalpha}
\bibliography{ref.bib}

\end{document}